\documentclass[12pt,a4paper]{article}
\usepackage[T1]{fontenc}
\usepackage[left=2.5cm, right=2.5cm, top=3cm, bottom=3cm]{geometry}

\usepackage{color}
\usepackage{amsmath} 
\usepackage{amsfonts} 
\usepackage{amsthm} 
\usepackage{amssymb}
\usepackage{mathrsfs}
\usepackage{mathtools}
\usepackage{enumitem}
\usepackage{babel}
\usepackage{hyperref}

\def\P{\mathbb{P}}

\def\R{{\mathbb{R}}}

\def\N{{\mathbb{N}}}
\def\Z{{\mathbb{Z}}}

\def\E{{\mathbb{E}}}

\def\d{\mathrm{d}}

\newcommand{\Lcal}   {{\mathcal L }}

\usepackage{amsmath}
\usepackage{relsize}
\newtheorem{theorem}{Theorem}

\newtheorem{proposition}[theorem]{Proposition}
\newtheorem{lemma}[theorem]{Lemma}
\newtheorem{corollary}[theorem]{Corollary}

\usepackage{colordvi}
\usepackage{amsfonts}
\usepackage{graphicx}
\usepackage{epsfig}
\usepackage{subcaption}
\usepackage{color}
\usepackage{mathtools, cuted}
\usepackage{url}
\usepackage{wrapfig}
\usepackage{tikz}
\usepackage{enumitem}

\newcommand{\1}{{\mathchoice {1\mskip-4mu\mathrm l}      
		{1\mskip-4mu\mathrm l}
		{1\mskip-4.5mu\mathrm l} {1\mskip-5mu\mathrm l}}}

\makeatletter
\renewcommand*{\p@section}{\S\,}
\renewcommand*{\p@subsection}{\S\,}
\renewcommand*{\p@subsubsection}{\S\,}
\makeatother

\usetikzlibrary{arrows,matrix,backgrounds,fit,calc,automata,shapes,positioning,decorations.pathreplacing}
\usetikzlibrary{calc}

\tikzset{
	between/.style args={#1 and #2}{
		at = ($(#1)!0.5!(#2)$)
	}
}

\begin{document}
	
	\title{Extinction Time Estimate for Subcritical (Binary) Homogeneous Crump-Mode-Jagers Processes}

	\author{
		Sima Mehri	}
	\maketitle
	\begin{abstract}
		Homogeneous Crump-Mode-Jagers processes   is considered with arbitrary initial population. The asymptotic distribution of the extinction time in  the limit of large initial population  is obtained in subcritical regime by assuming the existence of Malthusian parameter.
	\end{abstract}

\section{Crump-Mode-Jagers Processes}
In \cite{crump1968general, crump1969general, 
	nerman1981convergence},   the Crump-Mode-Jagers Process is described as follows: 
Let $I$ be the set of all individuals  indexes. Assume one ancestor indexed by $0$ is born at time $0$ and each individual with index $i\in I$ is born at some time $\tau _i \geq 0$ and lives for a random lifetime $\tau _i > 0$ with distribution function $F, (F(t):=\P(\tau_i\leq t))$. During its life, individual $i$ produces offspring at random times according to a random  point process $\xi_i$ (i.e. locally finite non-negative integer-valued Borel measure) on $[0,\tau_i)$, measured relative to its own birth time.   
Assume for every $i\in I$,  $(\xi_i, \tau _i)$ is independent of $\tau _i$, and the reproduction-death processes $\{(\xi_i, \tau _i)\}$ are i.i.d.\ copies of some generic reproduction process $(\xi,\tau)$.
The population counting process is defined by
$
Z_t $
the number of individuals alive at time $t$.
Let $Z_t(r)$ be the number of individuals at time $t$ with an age less than or equal to $r$.
Crump \& Mode \cite[Theorem 6.1]{crump1968general} showed the expected population size satisfies a renewal-type equation as follows: 
If $m(t) := \E[\xi([0,t])]$ denotes the expected number of births an individual produces up to age $t$, then
$
M(t) := \E[Z_t]
$
obeys the renewal equation
\[
M(t) = 1 - F(t) + \int_0^t M(t-s) \d m(s).
\]
Taking Laplace transforms, $\Lcal(M)(s):=\int_0^\infty M(t)e^{-st}\d t$,
\[
\Lcal(M)(s) = \frac{1-\E[e^{-s\tau }]}{s\left(1 - \Lcal(m)(s)\right)}.
\]
The exponential growth of the process is determined by the Malthusian parameter $\alpha $, which is the unique solution (if it exists) to
\[
\Lcal(m)(\alpha)=\int_0^\infty e^{-\alpha t} \d m(t) = 1.
\]

Let  $\xi:=\xi((0,\infty])$ be the number of offspring generated by an individual and  $m:=m(\infty)=\E[\xi((0,\infty])]$ be its mean.  Similar to Galton-Watson and Bellman-Harris processes, the CMJ processes are also classified in supercritical, critical, subcritical or
explosive cases according to the mean offspring number $m$,  as $1 <m< \infty$, $m= 1$, $m < 1$, or $m = \infty$, respectively.

\paragraph{Subcritical Case ($m<1$):} Assume Malthusian parameter $\alpha<0$ exists. According to \cite[page 132]{haccou2005branching}, under quite natural conditions, including $\E[\xi\log(\xi)] <\infty$, it holds true that
\begin{equation}\label{Equ-subCMJ-goal}
	\P(T>t\mid Z_0=Z)=\P(Z_t>0\mid Z_0=Z)\sim cZe^{\alpha t}, \qquad \text{as } t\to \infty
\end{equation}
for some constant $c\in (0,\infty)$.   

\paragraph{Critical Case ($m=1$):}

Assume $m=1$, $\sigma^2:=\mathrm{Var}[\xi]=\E[\xi^2]-1<\infty$,  $a:=\int_0^\infty t\d m(t)>0 $, and $t^2(1-m(t))\to 0$,  $t^2(1-F(t))\to 0$ as $t\to \infty$. Then by \cite[page 304]{holte1974extinction}, there holds
\[\lim_{t\to \infty} t\,\P(T>t)=\lim_{t\to \infty} t\,\P(Z_t>0)=2a\sigma^{-2}. \] 
\paragraph{Supercritical Case ($m(\infty)>1$):} Assume measure $ \int_\cdot\d m$ is non-lattice,  Malthusian parameter $\alpha\in (0,\infty)$ exists, and $\int_0^\infty te^{-\alpha t}\d m(t) <\infty$.  \cite[Corrollary 3.2]{nerman1981convergence} states under these assumptions, for all $r\in (0,\infty]$, 
\[e^{-\alpha t}Z_t(r)\to \frac{\int_0^r e^{-\alpha t}(1-F(t)) \d t}{\int_0^\infty te^{-\alpha t}\d m(t)}\cdot V \text{ in probability, as  } t\to \infty.\]
where $\E[V]=1$ and $ \P(V = 0) = \P(T<\infty) < 1$ is the extinction probability.

Doney \cite{doney1972Age} considered the 
special case of CMJ process where given $\tau $, the point process $\xi$ on $[0,\tau)$ is  inhomogeneous Poisson point process with rate $\lambda(x)$ for $x\in [0,\tau )$.

Define 
\[h(s):=\int_0^\infty e^{(s-1)\int_0^u \lambda (v) \d v}\d F(u)\]
and let $q$ be the smallest root $s\in[0,1]$ of $h(s)=s$. Then \cite[Theorem 4.1]{doney1972Age} states if $q=0$, $\forall t\geq 0: \P(T\leq t)=0$; if $q>0$, $\P(T\leq t)\uparrow q$ as $t\to \infty$. Therefore, $\P(T<\infty)=q$ the smallest root $s\in[0,1]$ of $h(s)=s$. 

\subsection{CMJ with Homogeneous Poisson  Reproduction Point Process}	
Lambert \cite{lambert2010contour} and Richard \cite[Chapter 2]{richard2011arbres} investigated splitting trees corresponding to the CMJ processes where given lifetime $\tau $, the point process $\xi$ on $[0,\tau)$ is homogeneous Poisson point process with constant rate $\lambda$. In the literature, this process  is called binary homogeneous CMJ process,   since at each birth time only single offspring  is born. They assumed the process starts by single individual at time $0$ with remaining lifetime random variable, say $\tau_R$, and denoted conditional  probability measure given $\tau_R=\chi, \chi\in \R$ by $\P_\chi$.   	The expected number of births an individual produces up to age $t$ is $$m(t)=\E[\xi([0,t])]=\E\left[\lambda\int_0^t \1_{\{s<\tau \}}\d s\right]=\lambda \int_0^t(1-F(s))\d s$$
Then Malthusian parameter $\alpha $ is the unique solution  (if it exists) to 
\begin{equation} \label{Equ-Malthusian}1=\int_0^\infty e^{-\alpha t}\d m(t).\end{equation}
There exists at most one solution to \eqref{Equ-Malthusian},  because the right hand side of \eqref{Equ-Malthusian} is strictly decreasing with respect to $\alpha$. 
We have 
\begin{align*}
	\int_0^\infty e^{-\alpha t}\d m(t)&=\lambda \int_0^\infty e^{-\alpha t}(1-F(t))\d t\intertext{ using integration by parts, we get}&= \frac{\lambda }{\alpha}\left[[-e^{-\alpha t}(1-F(t))]\mid_{0}^\infty- \int_0^\infty e^{-\alpha t}\d F(t)\right]\\&=\frac{\lambda }{\alpha}\left[1-\E[e^{-\alpha \tau }]\right].
\end{align*}
So Malthusian parameter satisfies equation
\begin{equation}\label{Equ-malthus}\lambda \E[e^{-\alpha \tau }]=\lambda-\alpha.\end{equation}

Three cases may happen: supercritical case $\lambda\E[\tau ]>1$, (then $\alpha>0$ exists \cite[page 33]{richard2011arbres})  or subcritical case $\lambda\E[\tau ]<1$  (then $\alpha<0$ may exist or not) or critical case $ \lambda\E[\tau]=1$ ($\alpha=0$). Proposition 5.6 in \cite{lambert2010contour} 
states the probability of extinction is
\begin{equation}
	\label{Equ-LambertExtPrChi} \P_\chi(T<\infty)=e^{-\alpha_+ \chi}
\end{equation}
and extinction time has distribution
\begin{equation}
	\label{Equ-LambertExtDistChi}
	\P_\chi(T\leq t)=\P_\chi(Z_t=0)=\frac{W(t-\chi)\1_{\{\chi\leq t\}}}{W(t)}.\end{equation}
where the function $W$ is strictly increasing and called the \emph{scale function} (\cite[page 11]{lambert2010contour}) and has Laplace transform
\[\Lcal(W)(s):=\int_{0}^\infty W(r)e^{-sr}\d r=\frac{1}{\psi(s)}\]
where
$$\psi(s):=s-\lambda \int_0^\infty (1-e^{-sr})\d F(r)=s-\lambda +\lambda \E[e^{-s\tau }],$$ defined in \cite[page 27]{lambert2010contour}.  By \cite[Lemma 2.12(iii)]{richard2011arbres}, 
 $W$ is differentiable with derivative $W^\prime$ satisfying
\begin{equation}\label{Equ-W-Wprime}
	\lambda W(x)-W^\prime (x)=\lambda\int_0^x W(x-y)\d F (y).
\end{equation}
Lambert \cite[Proposition 5.6]{lambert2010contour} also states conditional on being non-zero, $Z_t$ has a geometric distribution with success probability $1/W(t)$, i.e.
$$\P_\chi(Z_t=k\mid Z_t>0)=\frac{1}{W(t)} \left(1-\frac{1}{W(t)}\right)^{k-1}, \qquad k\in \N:=\{1,2,3,\ldots\}.$$

According to \cite[Proof of Proposition 5.8]{lambert2010contour} for subcritical ($\alpha<0$) or supercritical cases ($\alpha>0$), there holds:
\begin{equation}\label{Equ-W-Lim}
	\lim_{t\to\infty }e^{-\alpha_+ t}W(t)=	\frac{1}{1-\lambda\E[\tau e^{-\alpha_+ \tau }]}\, .
\end{equation}
and in critical case 
\[\lim_{t\to \infty }\frac{W(t)}{t}=\frac{2}{ \psi^{\prime\prime}(0+)}\]
where $\psi^{\prime\prime}$ is second derivative of $\psi$. .

\paragraph{Subcritical Case ($\lambda\E[{\tau}] <1$):} Assume $m:=\lambda\E[{\tau}]<1$. \cite[Proposition 5.8 (i)]{lambert2010contour} states
$$\lim_{t\to \infty}\P(Z_t=k\mid Z_t>0)=(1-m )m^{k-1}, \qquad k\in \N.$$

\paragraph{Critical Case ($\lambda\E[{\tau}] =1$):} Provided $\lambda\E[{\tau}] =1$ and $\lambda \E[\tau^2]=\lambda \int_0^\infty r^2\d F(r)<\infty$, \cite[Proposition 5.8 (i)]{lambert2010contour} states
$$\forall x\geq 0 :\qquad \lim_{t\to \infty}\P(Z_t/t\geq x\mid Z_t>0)=e^{-\psi^{\prime\prime}(0+)x/2},  $$
where $\psi(s):=s-\lambda \int_0^\infty (1-e^{-sr})\d F(r)=s-\lambda +\lambda \E[e^{-s\tau }]$  defined on \cite[page 27]{lambert2010contour}.
\paragraph{Supercritical Case ($\lambda\E[\tau ] >1$):} 
Assume population starts with one ancestor born at time zero, i.e. $Z_0=\1_{[0,\infty)}$. If $\lambda\E[\tau ] >1$ then 
according to \cite[Proposition 2.1(i) with assumption made on page 33]{richard2011arbres},  
$\P(T<\infty)=\E[e^{-\alpha \tau }]=1-\alpha/\lambda$
and conditional on $\{T=\infty\}$, 
\[e^{-\alpha t}Z_t\to \mathrm{Exp}(c)\quad \text{ almost surely as } \ t\to \infty,\]
where $c:=\psi^\prime (\alpha)=1-\lambda \E[\tau e^{-\alpha \tau }]$ and $\mathrm{Exp}(c)$ is an exponential random variable with parameter $c$.

In this paper, we derive the Gumbel  distribution of extinction time estimate error for large initial population  in Theorem \ref{Thm-Gumbel Limit}. 

\section{Main Results for Binary Homogeneous CMJ  Processes}
Consider an age-dependent  birth-death process where each individual has a constant birth rate  $\lambda$ and a lifetime, denoted by the random variable $\tau $ with a differentiable cumulative distribution function $F(x)=\P(\tau \leq x)$.  The instantaneous death rate, or hazard rate, for an individual of age $x$ is given by:
\[\mu(x):=\frac{F^\prime(x)}{1-F(x)},\]
with notation $F^\prime$ as derivative of function $F$. 
We assume that $x\mapsto \mu(x)\in \R$ is a continuous function. 
Let $Z_t$ be the number of alive individuals at time t, and $Z_t(r)$ be the number of alive individuals at time $t$ with age less than or equal to $r$. 
We also consider the number of individuals strictly younger than age $r$ by taking the limit $Z_t(r^-):=\lim_{s\nearrow r}Z_t(s)$.  Let $\int_0^\infty f(r)\d Z_t(r) $ denote the Lebesgue-Stieltjes  integral of a measurable function $f$  with respect to monotone and right-continuous function $Z_t(\cdot)$. Since $Z_t(\cdot)$ is non-decreasing integer-valued and right-continuous. We have
\begin{equation}\label{Equ-intdef}\int_0^\infty f(r)\d Z_t(r)=\sum_{\{x\geq 0: Z_t(x)\neq Z_t(x^-)\}}f(x)\left(Z_t(x)- Z_t(x^-)\right)=\sum_{i=1}^{Z_t}f(R_i(t)).\end{equation}
where the random variables
$$R_i(t):=\inf\{r\geq 0: Z_t(r)\geq i\}, \quad  i\in \N$$ 
are the ordered population ages   at time $t$.

Let $T$ be extinction time, 
\[
T := \inf \{ t \geq 0 : Z_t = 0 \},
\]
with $T=\infty$ if extinction never occurs. Define  $$q (x,t):= \P(T\leq    t\mid Z_0=\1_{[x,\infty)})= \P(Z_t=0\mid Z_0=\1_{[x,\infty)})$$ which is distribution function of extinction time given only one individual exists at time zero with age $x$. 
Since the extinction time for each person's lineage is independent of the others, we have by equality \eqref{Equ-intdef} that
\[\begin{aligned}\P(T\leq t)&=\E\left[\Pi_{1\leq i\leq Z_0}(\P(T\leq t\mid  Z_0=\1_{[R_i(0),\infty)}))\right]=\E\left[\Pi_{1\leq i\leq Z_0}q(R_i(0),t)\right]\\&=\E\left[\exp\left(\sum_{i=1}^{Z_0}\log(q(R_i(0),t)) \right)\right]=\E\left[\exp\left(\int_0^\infty\log\left(q\left(x,t\right)\right) \d Z_0( x)\right)\right].\end{aligned}\]

\begin{proposition}
	\label{Thm-Solution}
	There holds
	\begin{equation}
		\begin{split}	q(x,t)&=\P(T\leq t\mid Z_0=\1_{[x,\infty)})=\frac{\int_x^{x+t} W(t+x-s)\d F(s)}{(1-F(x))W(t)}\\&=1+\frac{\lambda \int_0^x (1-F(y))W^\prime(t+x-y)\d y -W^\prime(t+x)}{\lambda(1-F(x))W(t)}.
		\end{split}
	\end{equation}
	where the function $W$ is the scale function  defined in \cite[page 11]{lambert2010contour} having Laplace transform
	\begin{equation}\label{Equ-LaplaceW}\Lcal(W)(s):=\int_{0}^\infty W(r)e^{-sr}\d r=\frac{1}{\lambda \left( \E[e^{-s\tau }]-1\right)+s}.\end{equation}
\end{proposition}
\begin{proof} If there exists single individual at time zero with age $x$, its remaining lifetime, $\tau_R$ has the following distribution function: 
 $$\P(\tau _R\leq r\mid Z_0=\1_{[x,\infty)})=\P(\tau \leq r+x\mid \tau \geq x)=\frac{F(r+x)}{1-F(x)}.$$
	By Equation \eqref{Equ-LambertExtDistChi}, there holds
	\begin{align*}	q(x,t)&=\P(T\leq t\mid Z_0=\1_{[x,\infty)})=\E[\P_{\tau _R}(T\leq t)\mid Z_0=\1_{[x,\infty)}]\\&=\E\left[\frac{\1_{\{\tau _R\leq t\}}W(t-\tau _R)}{W(t)}\mid Z_0=\1_{[x,\infty)}\right]\\&=\int_0^\infty \1_{\{r\leq t\}}\frac{W(t-r)}{W(t)(1-F(x))}\d F(r+x)=\frac{\int_x^{x+t} W(t+x-s)\d F(s)}{(1-F(x))W(t)}\intertext{by equality \eqref{Equ-W-Wprime}, $ W^\prime(t+x)=\lambda W(t+x)-\lambda \int_0^{x+t} W(t+x-y)\d F(y)$, so we get }&=\frac{1}{\lambda(1-F(x))W(t)}\left[\lambda W(t+x)-\lambda \int_0^x W(t+x-y)\d F(y) -W^\prime(t+x)\right]\intertext{by integration by parts with functions $u(y):=\lambda W(t+x-y), v(y):=(1-F(y))$,  we get $\int_0^x \lambda  W(t+x-y)\d F(y)=-\int_0^xu(y)\d v(y)=-u(x)v(x)+u(0)v(0)+\int_0^x u^\prime(y) v(y)\d y=\lambda W(t+x)-\lambda W(t)(1-F(x))-\lambda \int_0^x (1-F(y))W^\prime(t+x-y)\d y$, so}&=1+\frac{\lambda \int_0^x (1-F(y))W^\prime(t+x-y)\d y -W^\prime(t+x)}{\lambda(1-F(x))W(t)}.
	\end{align*}
\end{proof}

In order to derive asymptotic behaviour for tail distribution of extinction time, we use the final value theorem which describes the asymptotic behaviour of a function according to its Laplace transform. Any complex number $a\in \mathbb{C}$ such that $\lim_{s\to a}\Lcal(f)(s)\notin \mathbb{C}$ is called pole of Laplace transform $\Lcal(f)$. The pole $a$ is called single if $\lim_{s\to a}(s-a)\Lcal(f)(s)<\infty$. 

We first derive the limit of $e^{-\alpha t}W^\prime(t)$ in Proposition \ref{Prop-Lim-W-Prime}, then we derive the limit for $e^{-\alpha t}(1-q(t,x))  $ and state it by multiplying to $(1-F(x))e^{-\alpha x}$ as a limit uniformly over $x\in \R_+$ in Proposition \ref{Cor-lim-g2}. By having the uniform limit, we derive extinction time estimate in Theorem \refeq{Thm-Gumbel Limit}.

\begin{theorem}[Standard Final Value Theorem] \cite{chen2007final}\label{Thm-FinalValue}
Assume for measurable function $f:\R_+\to \R$,  every pole of the Laplace transform $\Lcal(f)$ 
is either in the open left half plane or at the origin, and that 
$\Lcal(f)$  has at most a single pole at the origin. Then  $\lim_{t\to\infty}f(t)$ exists and $$\lim_{t\to\infty}f(t)=\lim_{s\to 0}s\Lcal(f)(s). $$
\end{theorem}

\begin{proposition}[Subcritical Case]\label{Prop-Lim-W-Prime}
Assume Malthusian parameter $\alpha<0$ exists, $\tau $ is non-lattice, and $\lambda\E[\tau e^{-\alpha  \tau }]\neq 1$. There holds
$$\lim_{t\to \infty}W^\prime(t) e^{-\alpha t}=
\frac{\alpha }{1-\lambda\E[\tau e^{-\alpha  \tau }]} \, .$$
Specially, it follows that	$\lim_{t\to \infty}W^\prime(t)=0$.
\end{proposition}
\begin{proof}
We use Theorem \ref{Thm-FinalValue} to prove the result. 
Let $f(t):=W^\prime (t)e^{-\alpha t}$.  By Equation \eqref{Equ-LaplaceW} and property of Laplace transform for function's derivative, 
\begin{align*}
\Lcal(f)(s)&=\int_0^\infty W^\prime(t)e^{(-\alpha-s) t}\d t=\Lcal(W^\prime)(s+\alpha)\\&=(s+\alpha)\Lcal(W)(s+\alpha)-W(0)\\&=\frac{\lambda(1-\E[e^{-(\alpha+s)\tau }])}{s+\alpha-\lambda(1-\E[e^{-(\alpha+s)\tau }])}\end{align*}
We claim that every pole of the Laplace transform $\Lcal(f)$ 
is either in the open left half plane or at the origin, and that 
$\Lcal(f)$  has at most a single pole at the origin. First, we observe  that  $s=-\alpha$ is not a pole of $\Lcal(f)(s)$, since by  L'H\^{o}pital rule,
\begin{align*}
\lim_{s\to -\alpha }\Lcal(f)(s)&=\lim_{s\to -\alpha}\frac{\lambda(1-\E[e^{-(\alpha+s)\tau }])}{s+\alpha-\lambda(1-\E[e^{-(\alpha+s)\tau }])}=\lim_{s\to -\alpha}\frac{\lambda\E[\tau e^{-(\alpha+s)\tau }]}{1-\lambda\E[\tau e^{-(\alpha+s)\tau }]}\\&=\frac{\lambda\E[\tau ]}{1-\lambda\E[\tau]}\in \mathbb{C}\end{align*}
which is finite value because $1>\lambda \E[\tau ]$ in subcritical case. 
The poles set of  $\Lcal(f)(s)$ is subset of
\[A:=\{s\in \mathbb{C}: \lambda-\alpha-s=\lambda \E[e^{-(\alpha+s)\tau }]\}\setminus \{-\alpha\}\]
where the denominator of $\Lcal(f)$ is zero.

Let $\mathrm{Re} (s)\in \R$ and $\mathrm{Im}(s)\in \R$ denote respectively the real part and imaginary part  of complex number $s\in \mathbb{C}$, i.e. $s=\mathrm{Re}(s)+\mathbf{i}\, \mathrm{Im}(s)$, $\mathbf{i}:=\sqrt{-1}$.  To show that the poles  of $\Lcal(f)$ belongs to open left half plane or at the origin, it is enough to show that if $s\in A$ then either $\mathrm{Re} (s)<0$ or $s=0$, or equivalently if $s\in A$ with $\mathrm{Re}(s)\geq 0$, then $s=0$. Assume $s\in A$ with $\mathrm{Re}(s)\geq 0$. So by definition of set $A$, there holds $\lambda -\alpha-s=\lambda \E[e^{-(\alpha+s)\tau }]$.  Since $\lambda, \alpha, \tau \in \R$, 
\begin{align*}
\lambda-\alpha-\mathrm{Re}(s)&=\mathrm{Re}(\lambda-\alpha-s)=\mathrm{Re}\left(\lambda\E\left[e^{-(\alpha+s)\tau }\right]\right)=\mathrm{Re}\left(\lambda\E\left[e^{-(\alpha+\mathrm{Re}(s)+\mathrm{Im}(s)\mathbf{i})\tau}\right]\right)\\&=\mathrm{Re}\left(\lambda\E\left[e^{-(\alpha+\mathrm{Re}(s))\tau}(\cos(\mathrm{Im}(s)\tau )-\mathbf{i}\, \sin(\mathrm{Im}(s)\tau ))\right]\right)\\&=	\lambda \E\left[e^{-(\alpha+\mathrm{Re}(s))\tau }\cos (\mathrm{Im}(s)\tau )\right]\\&\leq 	\lambda \E\left[e^{-(\alpha+\mathrm{Re}(s))\tau }\right].\end{align*}
In the last inequality, we used $\cos (\mathrm{Im}(s)\tau )\leq 1$. We want to show that  $s=0$. In order to use Holder inequality, we define $X$ to be exponential random variable with parameter $\lambda$, independent of $\tau$. Then it follows that
$\E\left[e^{(\alpha+\mathrm{Re}(s))X}\right]=\lambda/(\lambda-\alpha-\mathrm{Re}(s))$ and so $\E\left[e^{(-\alpha-\mathrm{Re}(s))(\tau -X)}\right]\geq 1$.

We consider three cases: 

\paragraph{Case $-\alpha>\mathrm{Re}(s)>0$:} We have $0<(-\alpha-\mathrm{Re}(s))/(-\alpha)<1$ and  by Holder  inequality,
\[1\leq  \E\left[e^{(-\alpha-\mathrm{Re}(s))(\tau -X)}\right] <\left(\E\left[e^{-\alpha(\tau -X)}\right]\right)^{(-\alpha-\mathrm{Re}(s))/(-\alpha) }\]
The last inequality is  strict, since $\tau -X$ is not a constant random variable. Hence $\E\left[e^{-\alpha(\tau -X)}\right]>1$ which contradicts with equation \eqref{Equ-malthus},  $\E\left[e^{-\alpha(\tau -X)}\right]= \lambda \E[e^{-\alpha \tau}]/(\lambda-\alpha)=1$. 

\paragraph{Case $-\alpha=\mathrm{Re}(s)$:} Since $s=-\alpha $ is not a pole, we have $\mathrm{Im}(s)\neq 0$. 
\begin{align*}
\lambda&=	\lambda-\alpha-\mathrm{Re}(s)=\mathrm{Re}(\lambda-\alpha-s)=\mathrm{Re}\left(\lambda\E\left[e^{-(\alpha+s)\tau }\right]\right)\\&=	\lambda \E\left[e^{-(\alpha+\mathrm{Re}(s))\tau }\cos (\mathrm{Im}(s)\tau )\right]=\lambda \E\left[\cos (\mathrm{Im}(s)\tau )\right]< 	\lambda.\end{align*}
which is contradiction. In the last strict inequality, we used non-lattice property of $\tau $ so that for the lattice $L:=\{2k\pi/\mathrm{Im}(s), k\in \N\} $, $\P(\tau \in  (\R\setminus L))>0$ and on $\{\tau \in (\R\setminus L)\}$, it holds $\cos(\mathrm{Im}(s)\tau)<1$, therefore $\E\left[\cos (\mathrm{Im}(s)\tau )\right]<1$. 

\paragraph{Case $-\alpha< \mathrm{Re}(s)$:}
Since for the function $\psi(r):=\E\left[e^{r(\tau -X)}\right]$, we have $$\psi^\prime(0)=\E[\tau -X]=\E[\tau ]-1/\lambda<0,$$ and  $$\psi''(r)=\E[(\tau-X)^2e^{r(\tau -X)}]\geq 0.$$ So for all $r<0$, $\psi^\prime(r)\leq \psi^\prime (0)<0$ and therefore  $1\leq \psi(-\alpha-\mathrm{Re}(s))< \psi(0)=1$ which contradicts with $\psi(-\alpha-\mathrm{Re}(s))=\E\left[e^{(-\alpha-\mathrm{Re}(s))(\tau -X)}\right]\geq 1$.

\paragraph{Case  $\mathrm{Re} (s)=0$ and $\mathrm{Im}(s)\neq 0$:} So  $s=\mathrm{Im}(s)\, \mathbf{i}$.  Then 
\[\lambda-\alpha=\mathrm{Re}\left(\lambda\E\left[e^{-(\alpha+\mathrm{Im}(s)\, \mathbf{i} )\tau }\right]\right)=\lambda\E\left[e^{-\alpha \tau }\cos(\mathrm{Im}(s)\tau )\right]<	\lambda\E\left[e^{-\alpha \tau }\right].\]
In the last strict inequality,  we used non-lattice property of $\tau $ so that for the lattice $L:=\{2k\pi/\mathrm{Im}(s), k\in \Z\} $, $\P(\tau \in  (\R\setminus L))>0$ and on $\{\tau \in (\R\setminus L)\}$, it holds $\cos(\mathrm{Im}(s)\tau)<1$. 
This contradicts with $\E\left[e^{-\alpha(\tau -X)}\right]= 1$. 

All above three cases lead to a contradiction.  Hence, every pole of the Laplace transform $\Lcal(f)$ 
is either in the open left half plane or at the origin. 

By  L'H\^{o}pital rule, we  get
\begin{align*}\lim_{s\to 0}s\Lcal(f)(s)&=\lim_{s\to 0}\frac{s\lambda}{s+\alpha-\lambda(1-\E[e^{-(\alpha+s)\tau }])}\times \lim_{s\to 0 }(1-\E[e^{-(\alpha+s)\tau }])\\&=\lim_{s\to 0}\frac{\lambda}{1-\lambda\E[\tau e^{-(\alpha+s)\tau }]} \times (1-\E[e^{-\alpha \tau }]) = \frac{\alpha}{1-\lambda\E[\tau e^{-\alpha \tau }]}\,. \end{align*} 

Since by assumption of Theorem, $\lambda \E[\tau e^{-\alpha \tau }]\neq 1$, therefore $\lim_{s\to 0}s\Lcal(f)(s)\in \R$ and  $\Lcal(f)$  has at most single pole at $s=0$. Hence, every pole of the Laplace transform $\Lcal(f)$ 
is either in the open left half plane or at the origin, and that 
$\Lcal(f)$  has at most a single pole at the origin. By Theorem  \ref{Thm-FinalValue}, $\lim_{t\to\infty}f(t)$ exists and $$\lim_{t\to\infty}f(t)=\lim_{s\to 0}s\Lcal(f)(s) $$ which yields
$$\lim_{t\to \infty}W^\prime(t) e^{-\alpha t}=
\frac{\alpha }{1-\lambda\E[\tau e^{-\alpha  \tau }]} \, .$$
\end{proof}
We use the following Lemma in the proof of next Theorem. 
\begin{lemma}\label{Lemma-MidEqulality}
Assume Malthusian parameter $\alpha\neq 0$ exists, i.e. there exists $\alpha\neq 0$ so that 
$\lambda-\alpha=\lambda\int_0^\infty e^{-\alpha y}\d F(y)$.
 There holds
\[\E\left[e^{-\alpha (\tau -x)}-1\mid \tau >x\right]=\frac{-\alpha e^{\alpha x}}{\lambda(1-F(x))}\left[-\lambda \int_0^x (1-F(y))e^{-\alpha y}\d y +1\right].\]
\end{lemma}
\begin{proof}
Using integration by parts for the functions $u(y)=-e^{-\alpha y}$ and $v(y)=1-F(y)$, we get 
\begin{align*}\int_0^x (1-F(y))(\alpha e^{-\alpha y})\d y &=\int_0^x u^\prime (y)v(y)\d y=u (x)v(x)-u(0)v(0)-\int_0^x u(y)\d v(y)\\&=  -(1-F(x))e^{-\alpha x}+1-\int_0^x e^{-\alpha y}\d F(y)
\end{align*}

So 
\begin{align*}\MoveEqLeft[2]\frac{-\alpha e^{\alpha x}}{\lambda(1-F(x))}\left[-\lambda \int_0^x (1-F(y))e^{-\alpha y}\d y +1\right]\\&=\frac{ e^{\alpha x}}{\lambda(1-F(x))}\left[-\lambda(1-F(x)) e^{-\alpha x}+\lambda -\alpha -\lambda\int_0^x e^{-\alpha y}\d F(y)\right]\intertext{	since $\lambda-\alpha=\lambda\int_0^\infty e^{-\alpha y}\d F(y)$,}&=\frac{e^{\alpha x} }{(1-F(x))}\left[ -(1-F(x))e^{-\alpha x}+\int_x^\infty e^{-\alpha y}\d F(y) \right]\\&=\E\left[e^{-\alpha (\tau -x)}-1\mid \tau >x\right].
\end{align*}

\end{proof}
\begin{proposition}\label{Cor-lim-g2} 
Assume Malthusian parameter $\alpha<0 $ exists,  $\tau $ is non-lattice, and \linebreak $\lambda\E[\tau e^{-\alpha  \tau }]\neq 1$. Let $p(x,t):=\P(T> t\mid Z_0=\1_{[x,\infty)})=1-q(x,t)$ denote the survival probability of binary homogeneous CMJ process at time $t$ starting with a single individual having age $x$ at time zero. There holds
\begin{itemize}
\item[(a)]
$\displaystyle\lim_{t\to \infty }p(x,t)e^{-\alpha t}
=\frac{1-\lambda\E[\tau ]}{\lambda \E[\tau e^{-\alpha \tau }]-1}\E\left[e^{-\alpha (\tau -x)}-1\mid \tau >x\right].$

\item[(b)]$\displaystyle\lim_{t\to\infty}\sup_{x\geq 0}\Bigg\lvert (1-F(x))p(x,t)e^{-\alpha (t+x)} -\frac{(1-\lambda \E[\tau ])\E\left[\left( e^{-\alpha \tau }-e^{-\alpha x}\right)\1_{ \{\tau >x\}}\right]}{\lambda\E[\tau e^{-\alpha \tau }]-1}\Bigg\rvert=0$.
\end{itemize}
\end{proposition}

\begin{proof}

(a)	Since $p(x,t)=1-q(x,t)$, by  Proposition \ref{Thm-Solution},  we have
\begin{equation}\label{Equ-p-in-terms-Wprime}
\begin{split}	p(x,t)=\frac{-\lambda \int_0^x (1-F(y))W^\prime(t+x-y)\d y +W^\prime(t+x)}{\lambda(1-F(x))W(t)}.
\end{split}
\end{equation}
So, we can write
\begin{align*}
\lim_{t\to \infty }p(x,t)e^{-\alpha t}&=\lim_{t\to \infty }\frac{e^{-\alpha  t}}{\lambda(1-F(x))W(t)}\left[-\lambda \int_0^x (1-F(y))W^\prime(t+x-y)\d y +W^\prime(t+x)\right]\intertext{By Proposition \ref{Prop-Lim-W-Prime} and Equation \eqref{Equ-W-Lim}, we have respectively
$\lim_{t\to \infty}W^\prime(t) e^{-\alpha t}=
\alpha /(1-\lambda\E[\tau e^{-\alpha  \tau }])$ and 	$\lim_{t\to \infty}W(t) =1/(1-\lambda \E[\tau ])$ which yield}&= \frac{-\alpha e^{\alpha x} (1-\lambda \E[\tau ])}{\lambda(\lambda\E[\tau e^{-\alpha \tau }]-1)(1-F(x))}\left[-\lambda \int_0^x (1-F(y))e^{-\alpha y}\d y +1\right]\intertext{By  equality $\E\left[e^{-\alpha (\tau -x)}-1\mid \tau >x\right]=\frac{-\alpha e^{\alpha x}}{\lambda(1-F(x))}\left[-\lambda \int_0^x (1-F(y))e^{-\alpha y}\d y +1\right]$ of Lemma \ref{Lemma-MidEqulality}, we get}&=\frac{1-\lambda \E[\tau ]}{\lambda\E[\tau e^{-\alpha \tau }]-1}\cdot \E\left[e^{-\alpha (\tau -x)}-1\mid \tau >x\right].
\end{align*}

(b) By Lemma \ref{Lemma-MidEqulality}, we have
$$\E\left[e^{-\alpha (\tau -x)}-1\mid \tau >x\right]=\frac{-\alpha e^{\alpha x}}{\lambda(1-F(x))}\left[-\lambda \int_0^x (1-F(y))e^{-\alpha y}\d y +1\right].$$
By multiplying both sides by $(1-F(x))e^{-\alpha x}$, we get
\begin{equation}\label{Mid-Calcul}\E\left[\left( e^{-\alpha \tau }-e^{-\alpha x}\right)\1_{ \{\tau >x\}}\right]=\frac{-\alpha }{\lambda}\left[-\lambda \int_0^x (1-F(y))e^{-\alpha y}\d y +1\right].\end{equation}

We have by equality \eqref{Mid-Calcul}, and Proposition \ref{Thm-Solution}, 
\begin{align*}
\MoveEqLeft[2]\lim_{t\to\infty}\sup_{x\geq 0}\left\lvert (1-F(x))p(x,t)e^{-\alpha (t+x)}-\frac{(1-\lambda \E[\tau ])\E\left[\left( e^{-\alpha \tau }-e^{-\alpha x}\right)\1_{ \{\tau >x\}}\right]}{\lambda\E[\tau e^{-\alpha \tau }]-1}\right\rvert\\ &=\lim_{t\to \infty }\sup_{x\geq 0}\Bigg\lvert \frac{e^{-\alpha (t+x)}}{\lambda W(t)}\left[-\lambda \int_0^x (1-F(y))W^\prime(t+x-y)\d y +W^\prime(t+x)\right]\\&\qquad \qquad\qquad  -\frac{\alpha(1-\lambda\E[\tau ])}{\lambda(1-\lambda \E[\tau e^{-\alpha \tau }])}\times\left[-\lambda \int_0^x (1-F(y))e^{-\alpha y}\d y +1\right]\Bigg\rvert \intertext{By  putting corresponding  terms together, we get}&\leq \lim_{t\to \infty }\sup_{x\geq 0}\Bigg\lbrace  \int_0^x (1-F(y))e^{-\alpha y}\left \lvert \frac{W^\prime(t+x-y)e^{-\alpha (t+x-y)}}{W(t)}-\frac{\alpha(1-\lambda \E[\tau ]) }{1-\lambda\E[\tau e^{-\alpha \tau }]}\right\rvert \d y\\&\qquad  +\frac{1}{\lambda } \left \lvert \frac{W^\prime(t+x)e^{-\alpha (t+x)}}{W(t)}-\frac{\alpha (1-\lambda\E[\tau ])}{1-\lambda\E[\tau e^{-\alpha \tau }]}\right\rvert\Bigg\rbrace \\&=\lim_{t\to \infty}\sup_{r\geq  t}\left\lvert \frac{e^{-\alpha r}W^\prime(r)}{W(t)}-\frac{\alpha (1-\lambda\E[\tau ]) }{1-\lambda \E[\tau e^{-\alpha \tau }]}\right\rvert\cdot  \left[ \int_0^\infty (1-F(y))e^{-\alpha y}\d y +\frac{1}{\lambda} \right]\intertext{By subtracting and adding the term $\alpha /(W(t)(1-\lambda \E[\tau e^{-\alpha \tau }]))$ inside absolute value of the first factor, and using equation of Malthusian parameter, i.e. $\int_0^\infty (1-F(y))e^{-\alpha y}\d y =1/\lambda $, for the second factor, we get}
&\leq   \frac{2}{\lambda}\lim_{t\to \infty}\sup_{r\geq  t}\left\lvert e^{-\alpha r}W^\prime(r)-\frac{\alpha  }{(1-\lambda \E[\tau e^{-\alpha \tau }])}\right\rvert \lim_{t\to \infty}\frac{1}{W(t)}\\&\quad +\frac{2}{\lambda}\lim_{t\to \infty}\left\lvert \frac{\alpha  }{W(t)(1-\lambda \E[\tau e^{-\alpha \tau }])}-\frac{\alpha (1-\lambda\E[\tau ]) }{1-\lambda \E[\tau e^{-\alpha \tau }]}\right\rvert
\intertext{By Equation \eqref{Equ-W-Lim}, $\lim_{t\to \infty }(1/W(t))=1-\lambda \E[\tau ]$, and by Proposition \ref{Prop-Lim-W-Prime},  $\lim_{r\to \infty } e^{-\alpha r}W^\prime(r)=\alpha/(1-\lambda \E[\tau e^{-\alpha \tau }])$ which implies $\lim_{t\to \infty}\sup_{r\geq  t}\left\lvert e^{-\alpha r}W^\prime(r)-\alpha /( 1-\lambda \E[\tau e^{-\alpha \tau }])\right\rvert=0 $. So,}&=0.
\end{align*}
\end{proof}

\begin{theorem}[Estimate of Extinction Time]\label{Thm-Gumbel Limit}
Consider a birth death process with non-lattice lifetime random variable $\tau $.  Assume that $\lambda \E[\tau ]<1$ (subcritical case) and \linebreak $\lambda\E[\tau e^{-\alpha  \tau }]\neq 1$. 
Assume Malthusian parameter $\alpha<0$ exists.  We consider a sequence of initial configurations  indexed by $k\in \N$. For each configuration $k$, the initial population is given by a  non-negative, random integer-valued measure $Z_k$ on $\R_+$, where $Z_k(r)$ represents the number of individuals with age in the interval $[0,r]$ at time zero.   
Assume that the total number of individuals at time zero, $M_k:=Z_k(\infty)$, tends to infinity almost surely as $k\to \infty$. Furthermore, assume that
\begin{equation}\label{Equ-Assum}
U:=\sup_{x\geq 0}\E[e^{-\alpha(\tau -x)}\mid \tau >x]<\infty
\end{equation}
and 
 the functions  $Z_k/M_k$ converge almost surely in  total variation distance  to  a non-decreasing random distribution  function  $B$, i.e. 
\[\lVert Z_k/M_k-B\rVert_{TV}:=\int_0^\infty \lvert  \d Z_k(  x)/M_k - \d B( x) \rvert \]
converges to zero almost surely  as $k\to \infty$.
Let
\[D:=\frac{1-\lambda \E[\tau]}{\lambda\E[\tau e^{-\alpha  \tau }]-1}\cdot\int_0^\infty\E\left[e^{-\alpha  (\tau -x)}-1\mid \tau >x\right] \d B(x)<\infty.\]
Let $T_k$ be the extinction time for configuration $k$. Then, as $k\to \infty$, the random variables $$-\alpha  T_k-\log( M_k)-\log(D)$$ converges in distribution to a Gumbel random variable $Z$ with distribution function $\P(Z\leq r)=e^{-e^{-r}}$. In other words, 
\[\lim_{k\to \infty}\P(-\alpha  T_k-\log( M_k)-\log(D)\leq r)=e^{-e^{-r}}.\]
\end{theorem}
\begin{proof}
Since the extinction time for each individual's lineage is independent of the others, we have
\[\begin{aligned}\P(T_k\leq t)&
=\E\left[\Pi_{\{x: Z_k(x)\neq Z_k(x^-)\}}(1-p(x,t))^{Z_k(x)-Z_k(x^-)}\right]\\&=\E\left[\exp\left(\sum_{\{x: Z_k(x)\neq Z_k(x^-)\}}\log (1-p(x,t))\left(Z_k(x)-Z_k(x^-)\right)\right)\right]\\&=\E\left[\exp\left(\int_0^\infty\log\left(1-p\left(x,t\right)\right)\d Z_k( x)\right)\right]\end{aligned}\]

Therefore,
\begin{align*}\MoveEqLeft[3]\lim_{k\to \infty}\P\left( -\alpha T_k\leq \log(M_k)+\log(D)+r\right)
\\&=\lim_{k\to \infty}\E\left[\exp\left(\int_0^\infty\log\left(1-p\left(x,\frac{\log(M_k)+\log(D)+r}{-\alpha }\right)\right) \d Z_k( x)\right)\right]\end{align*}
Since inside above expectation is less than $1$, to prove Theorem, by Dominated Convergence Theorem, we have to show the following limit is zero:		
\begin{equation}\label{Equ-SplitLimit}\begin{split}
&\lim_{k\to \infty }\left\lvert\int_0^\infty\log\left(1-p\left(x,\frac{\log(M_k)+\log(D)+r}{-\alpha}\right)\right)\d  Z_k( x)+e^{-r}\right\rvert\\&\quad \leq \lim_{k\to \infty }\left\lvert\int_0^\infty M_k\log\left(1-p\left(x,\frac{\log(M_k)+\log(D)+r}{-\alpha}\right)\right) (\d Z_k( x)/M_k-\d B( x))\right\rvert\\&\qquad +\lim_{k\to \infty }\left\lvert\int_0^\infty M_k\log\left(1-p\left(x,\frac{\log(M_k)+\log(D)+r}{-\alpha}\right)\right)\d  B(x)+e^{-r}\right\rvert\end{split}
\end{equation}
 In parts 1,2 we show that first and second limit on the right hand side of \eqref{Equ-SplitLimit} is zero. 
 
\paragraph{Part 1: }First we show  the first limit on the right-hand side of \eqref{Equ-SplitLimit} is zero.  Since $\d Z_k/M_k$ converges to $\d B$ in   total variation distance, in order to show the first term in the right hand side of \eqref{Equ-SplitLimit} converges to zero, it is enough to show that there exists $K\in \N$ so that for all $x\geq 0$, $k\geq K$,  $$M_k\left\lvert \log\left(1-p\left(x,\frac{\log(M_k)+\log(D)+r}{-\alpha}\right)\right)\right\rvert\leq C$$
for some constant $C<\infty$.  Since $\lim_{k\to \infty }M_kD=\infty$, Proposition \ref{Cor-lim-g2} (b) implies for every arbitrary $\varepsilon>0$, and  $C_1:=\frac{(1-\lambda \E[\tau ])}{\lambda\E[\tau e^{-\alpha \tau }]-1}$,  there exists an integer $K$ such that for all  $k\geq K$, \begin{equation}\begin{split} M_kDe^{r}p\left(x,\frac{\log(M_k)+\log(D)+r}{-\alpha}\right)(1-F(x))e^{-\alpha x}&\leq \varepsilon+C_1\E[(e^{-\alpha \tau }-e^{-\alpha x})\1_{\{\tau>x\}}].\end{split}
	\end{equation}
	So,
	\begin{equation*}\begin{split} M_kDe^rp\left(x,\frac{\log(M_k)+\log(D)+r}{-\alpha}\right)&\leq  \varepsilon e^{\alpha x}/(1-F(x))+C_1\E[e^{-\alpha (\tau-x) }\mid \tau>x]\\&\leq (C_1+\varepsilon) \E[e^{-\alpha (\tau-x) }\mid \tau>x]\leq U(C_1+\varepsilon)=:C_2.\end{split}
	\end{equation*}
Since $\log(1-a)\leq a/(1-a)$, we get for $k\geq K$, 
  \begin{align}\label{Equ-BDD-Mlog}M_k\left\lvert \log\left(1-p\left(x,\frac{\log(M_k)+\log(D)+r}{-\alpha}\right)\right)\right\rvert\nonumber&\leq \frac{M_kp\left(x,\frac{\log(M_k)+\log(D)+r}{-\alpha}\right)}{1-p\left(x,\frac{\log(M_k)+\log(D)+r}{-\alpha}\right)}\\\nonumber & \leq \frac{C_2D^{-1}e^{-r}}{1-C_2 D^{-1}M_k^{-1}e^{-r}}\intertext{Since $M_k\to \infty $ as $k\to \infty$, there exists $K_2$ so that for  $k\geq K_2\vee K$, there holds $M_k^{-1}D^{-1}e^{-r}C_2\leq  1/2$. Therefore }&\leq C_2D^{-1}e^{-r}/2.\end{align}
Hence,
\begin{align*}\MoveEqLeft[2]\lim_{k\to \infty }\left\lvert\int_0^\infty M_k\log\left(1-p\left(x,\frac{\log(M_k)+\log(D)+r}{\theta}\right)\right) (\d Z_k( x)/M_k-\d B( x))\right\rvert\\&\leq e^{-r}D^{-1}(C_2/2)\lim_{k\to \infty }\int_0^\infty  \lvert \d Z_k( x)/M_k-\d B( x)\rvert\\&= e^{-r}D^{-1}(C_2/2) \lim_{k\to \infty}\lVert Z_k/M_k-B\rVert_{TV}=0.\end{align*}
\paragraph{Part 2:} For the second limit on the right-hand side of \eqref{Equ-SplitLimit},  by replacing $e^{-r}$ with  $$\frac{D^{-1}e^{-r}(1-\lambda \E[\tau])}{\lambda\E[\tau e^{\theta \tau }]-1}\cdot\int_0^\infty\E\left[e^{-\alpha (\tau -x)}-1\mid \tau >x\right] \d B( x)=e^{-r},$$
we get
\begin{align*}
&\lim_{k\to \infty }\left\lvert\int_0^\infty M_k\log\left(1-p\left(x,\frac{\log(D)+\log(M_k)+r}{-\alpha}\right)\right)\d B( x)+e^{-r}\right\rvert\\&=\lim_{k\to \infty }\Bigg\lvert\int_0^\infty M_k\log\left(1-p\left(x,\frac{\log(D)+\log(M_k)+r}{-\alpha}\right)\right) \d B( x)\\&\qquad \qquad+D^{-1}e^{-r}\frac{1-\lambda \E[\tau]}{\lambda\E[\tau e^{-\alpha \tau }]-1}\cdot\int_0^\infty\E\left[e^{-\alpha (\tau -x)}-1\mid \tau >x\right] \d B( x)\Bigg\rvert\\&\leq \lim_{k\to \infty}\frac{1}{De^r}\int_0^\infty \Bigg\lvert DM_ke^{r}\log\left(1-p\left(x,\frac{\log(D)+\log(M_k)+r}{-\alpha}\right)\right)\\&\qquad \qquad +\frac{(1-\lambda \E[\tau])\E\left[e^{-\alpha (\tau -x)}-1\mid \tau >x\right]}{\lambda\E[\tau e^{-\alpha \tau }]-1}\Bigg\rvert\d B( x)\intertext{The  inequalities \eqref{Equ-Assum}, \eqref{Equ-BDD-Mlog} imply the integrand in the above integral is bounded and we can apply Dominated Convergence Theorem,}
&=\frac{1}{De^r}\int_0^\infty
\lim_{k\to \infty} \Bigg\lvert DM_ke^{r}\log\left(1-p\left(x,\frac{\log(D)+\log(M_k)+r}{-\alpha}\right)\right)\\&\qquad \qquad + \frac{(1-\lambda \E[\tau])\E\left[e^{-\alpha (\tau -x)}-1\mid \tau >x\right]}{\lambda\E[\tau e^{-\alpha \tau }]-1}\Bigg\rvert \d B(x)
\end{align*}
Since $B(\infty)=1$, for part  2, it is enough to show that limit inside above integral is zero. We have
\begin{align*}& \lim_{k\to \infty} \left\lvert DM_ke^{r}\log\left(1-p\left(x,\frac{\log(D)+\log(M_k)+r}{-\alpha}\right)\right)+\frac{(1-\lambda \E[\tau])\E\left[e^{-\alpha (\tau -x)}-1\mid \tau >x\right]}{\lambda\E[\tau e^{-\alpha \tau }]-1}\right\rvert\\&=\lim_{k\to \infty}\Bigg\lvert DM_ke^{r}p\left(x,\frac{\log(D)+\log(M_k)+r}{-\alpha}\right)\frac{\log\left(1-p\left(x,\frac{\log(D)+\log(M_k)+r}{-\alpha}\right)\right)}{p\left(x,\frac{\log(D)+\log(M_k)+r}{-\alpha}\right)}\\&\qquad \qquad+\frac{(1-\lambda \E[\tau])\E\left[e^{-\alpha (\tau -x)}-1\mid \tau >x\right]}{\lambda\E[\tau e^{-\alpha \tau }]-1}\Bigg\rvert\\&=\lim_{t\to \infty}\left\lvert e^{-\alpha t}p(x,t)\frac{\log (1-p(x,t))}{p(x,t)}+\frac{(1-\lambda \E[\tau])\E\left[e^{-\alpha (\tau -x)}-1\mid \tau >x\right]}{\lambda\E[\tau e^{-\alpha \tau }]-1}\right\rvert\intertext{ Proposition \ref{Cor-lim-g2} (a) implies $\lim_{t\to \infty}p(x,t)=0$  and  since $\lim_{\delta \to 0}\log(1-\delta)/\delta=-1$, $\lim_{t \to \infty }\log(1-p(x,t))/p(x,t)=-1$. So,}\\&=\lim_{t\to \infty}\left\lvert e^{-\alpha t}p(x,t)-\frac{(1-\lambda \E[\tau])\E\left[e^{-\alpha (\tau -x)}-1\mid \tau >x\right]}{\lambda\E[\tau e^{-\alpha \tau }]-1}\right\rvert=0.\end{align*}
The last equality is obtained by Proposition   \ref{Cor-lim-g2} (a). 
\end{proof}

\begin{corollary}[Estimate of Extinction Time in  Exponential Ages Distribution]\label{Cor-Gumbel Limit}
	Consider a birth death process with birth and death rates $\lambda$ and $\mu$.  Assume that $\lambda /\mu<1$ (subcritical case). 
	Assume Malthusian parameter $\alpha=\lambda-\mu$ exists.  We consider a sequence of initial configurations  indexed by $k\in \N$. For each configuration $k$, the initial population is given by a  non-negative, random integer-valued measure $Z_k$ on $\R_+$, where $Z_k(r)$ represents the number of individuals with age in the interval $[0,r]$ at time zero.   
	Let
	\[D:=1-\lambda /\mu.\]
	Let $T_k$ be the extinction time for configuration $k$. Then, as $k\to \infty$, the random variables $$-\alpha  T_k-\log( M_k)-\log(D)$$ converges in distribution to a Gumbel random variable $Z$ with distribution function $\P(Z\leq r)=e^{-e^{-r}}$. In other words, 
	\[\lim_{k\to \infty}\P(-\alpha  T_k-\log( M_k)-\log(D)\leq r)=e^{-e^{-r}}.\]
\end{corollary}

\begin{proof}
Since the exponential random variable is memoryless, the ages of individuals at time zero are not affecting the future of process and there is no need that  functions  $Z_k/M_k$ converge almost surely in  total variation distance  to  a non-decreasing random distribution  function  $B$. So, Theorem \ref{Thm-Gumbel Limit} implies the corollary.
\end{proof}
Consider a sequence of  binary birth-death processes indexed by $N$, with population size $Z_N(t)$ at time $t$, birth rate $\lambda_N$, and death rate $\mu_N$ for $N \in \{1,2,\ldots\}$.  
Let $T_N$ denote the extinction time,
\[
T_N:= \inf\{t \ge 0 : Z_N(t) = 0\},
\]
with the convention $\inf(\emptyset) = \infty$.  
If $Z_N(0)(\mu_N - \lambda_N) \to \infty$ as $N \to \infty$, then, according to \cite[Page 469]{brightwell2018extinction},
\[
(\mu_N - \lambda_N) T_N - \left( \log Z_N(0) + \log(\mu_N - \lambda_N) - \log \mu_N \right) \to Z,
\]
in distribution as $N\to \infty$, where $Z$ is the standard Gumbel random variable which is consistent with Corollary \ref{Cor-Gumbel Limit}.
\paragraph{Acknowledgement} The author is deeply grateful to Prof. Malwina Luczak and Prof. Andrew Hazel for their constructive feedback and continuous support, which have \linebreak significantly enhanced this study. She also acknowledges the use of AI-assisted tools (Microsoft Copilot, Google Gemini, OpenAI ChatGPT, and Perplexity) for language refinement and for improving the clarity and comprehensiveness of the literature review.
\bibliographystyle{plain}
\bibliography{epid}

\end{document}